\documentclass{amsart}
\usepackage{amsmath, amssymb,amsthm,latexsym}
\usepackage{enumerate}

\newtheorem{theorem}{Theorem}[section]

\theoremstyle{definition}
\newtheorem{example}{Example}

\newtheorem{question}[theorem]{Question}

\newtheorem{remark}[theorem]{Remark}
\newtheorem{problem}[theorem]{Open Problem}

\title{Monoids $\mathrm{Mon}\langle a,b:a^{\alpha}b^{\beta}a^{\gamma}b^{\delta}=b\rangle$ admit finite complete rewriting systems}

\author{Alan J. Cain}

\address{Centro de Matematica, Universidade do Porto, Rua do Campo Alegre 687,
4169--007 Porto, PORTUGAL}

\email{\texttt{ajcain@fc.up.pt}}

\author{Victor Maltcev}

\address{Department of Mathematics and Statistics, Sultan Qaboos University, Al-Khodh 123, Muscat, 
Sultanate of OMAN}

\email{\texttt{victor.maltcev@gmail.com}}

\begin{document}

\begin{abstract}
We prove that every monoid $\mathrm{Mon}\langle
a,b:a^{\alpha}b^{\beta}a^{\gamma}b^{\delta}=b\rangle$ admits a finite
complete rewriting system. Furthermore we prove that
$\mathrm{Mon}\langle a,b:ab^2a^2b^2=b\rangle$ is non-hopfian,
providing an example of a finitely presented non-residually finite
monoid with linear Dehn function.
\end{abstract}


\maketitle

\section{Introduction}

The solubility of the word problem for one-relator monoids is a
long-standing open question. In a series of papers by Sergei Adian and
his students it was proved that the word problem for one-relator
monoids can be reduced to the cases $\mathrm{Mon}\langle
a,b:aUb=bVb\rangle$ and $\mathrm{Mon}\langle a,b:aUb=b\rangle$; we
refer the reader to the very nice survey~\cite{Adian} and references
therein. The methods of Adian's school is mostly combinatorics on
words, and sometimes the proofs using these methods can become quite
technically involved. On the other hand, Louxin Zhang showed
in~\cite{Zhang} how powerful the tools of rewriting systems can be in
trying to prove that the word problem for one-relator semigroups is
decidable. A remarkable paper of Yuji Kobayashi~\cite{Kobayashi}
showed that every one-relator monoid satisfies the condition FDT, and
since every monoid presented by a finite complete rewriting system
satisfies FDT, it prompted Kobayashi to ask:

\begin{problem}
Does every one-relator monoid admit a finite complete rewriting system?
\end{problem}

The aim of this note is to show that monoids $\mathrm{Mon}\langle
a,b:a^{\alpha}b^{\beta}a^{\gamma}b^{\delta}=b\rangle$ admit finite
complete systems, see Section~\ref{sec:cs}. Notice that these monoids
fall within one of the two important classes identified by Adian's
school. After that, in Section~\ref{sec:hopf}, we will prove that
$\mathrm{Mon}\langle a,b:ab^2a^2b^2=b\rangle$ is non-hopfian. This
gives an example of a non-residually finite finitely presented monoid
with linear Dehn function. This is significant because the analogous
question for finitely presented groups with linear Dehn function,
which are of course the hyperbolic groups, is an important open
problem. Finally, in Section~\ref{sec:final} we will state our
feelings about general monoids $\mathrm{Mon}\langle a,b:aUb=b\rangle$
and pose some questions.

\section{Preliminaries}

By a \emph{rewriting system} $(A,R)$ we mean a finite alphabet $A$ and a subset $R\subseteq A^{\ast}\times A^{\ast}$, where $A^{\ast}$ stands for the free monoid over $A$. Every pair $(l,r)$ from $R$ is called a \emph{rule} and normally is written as $l\to r$. For $x,y\in A^{\ast}$ we write $x\to y$, if there exist $\alpha,\beta\in A^{\ast}$ and a rule $l\to r$ from $R$ such that $x=\alpha l\beta$ and $y=\alpha r\beta$. Denote by $\to^{\ast}$ the transitive reflexive closure of $\to$. A rewriting system $(A,R)$ is called
\begin{itemize}
\item
\emph{confluent} if for every words $w,x,y\in A^{\ast}$ such that $w\to^{\ast}x$ and $w\to^{\ast}y$, there exists $W\in A^{\ast}$ such that $x\to^{\ast}W$ and $y\to^{\ast}W$;
\item
\emph{terminating} if there is no infinite derivation $x_0\to x_1\to x_2\to\cdots$.
\end{itemize}
Confluent terminating rewriting systems, which are also called
\emph{complete systems}, give a very convenient way of working with
finitely generated monoids. For, if a monoid is presented by
$M=\mathrm{Mon}\langle A:l_i=r_i\quad i\in I\rangle$ and it turns that
$S=(A,\{l_i\to r_i\}_{i\in I})$ is complete, then the elements of $M$
are in bijection with the \emph{normal forms} for $S$, i.e. those
words from $A^{\ast}$ which do not include any subword $l_i$, and to
find the normal form for a word $w\in A^{\ast}$, we just need to apply
the relation $\to$ successively to $w$ as many times as we can (this
process must stop by the termination condition) and the result will
always be the same word depending only on the element of $M$ that $w$
represents.

We refer the reader to the monograph of Ronald Book and Friedrich
Otto~\cite{Otto} for more background information on rewriting systems.

Let us provide our two final definitions. Let $\mathrm{Mon}\langle
A:R\rangle$ be a finite presentation for a monoid $M$. For two words
$x,y\in A^{\ast}$, equal in $M$, denote by
\begin{itemize}
\item
$d(x,y)$ the minimal number of relations from $R$ that need to be applied
  to obtain $x$ from $y$.
\item
$s(x,y)$ the least possible value of $\sup\{|w_i|:0\leq i\leq k\}$ for
  all derivations $x=w_0\sim w_1\sim\cdots\sim w_k=y$, where $p\sim q$
  stands for applying a single relation from $R$.
\end{itemize}
Then
\begin{equation*}
\mathbf{d}_n(M)=\sup\{d(x,y):x,y\in A^{\ast},~x=_My,~|x|_A,|y|_A\leq n\}
\end{equation*}
is called the \emph{Dehn function} of $M$, and
\begin{equation*}
\mathbf{sp}_n(M)=\sup\{s(x,y):x,y\in A^{\ast},~x=_My,~|x|_A,|y|_A\leq n\}
\end{equation*}
is called the \emph{space function} of $M$.

\section{Finite Complete Systems}\label{sec:cs}

\begin{theorem}\label{th:abab}
Every monoid $M=\mathrm{Mon}\langle a,b:a^{\alpha}b^{\beta}a^{\gamma}b^{\delta}=b\rangle$ admits a finite complete system.
\end{theorem}

\begin{proof}
If there are no overlaps of the word
$a^{\alpha}b^{\beta}a^{\gamma}b^{\delta}$ with itself, then
\begin{eqnarray*}
a^{\alpha}b^{\beta}a^{\gamma}b^{\delta} &\to& b
\end{eqnarray*}
is a complete rewriting system for $M$. The word
$a^{\alpha}b^{\beta}a^{\gamma}b^{\delta}$ only overlaps with itself
when $\beta\geq\delta$ and $\gamma\geq \alpha$. Thus we may
assume that $a^{\alpha}b^{\beta}a^{\gamma}b^{\delta}\equiv
a^{p}b^{q+s}a^{r+pk}b^{s}$ where $p,s,k\geq 1$, $q\geq 0$ and $0\leq
r<p$.

\medskip

\noindent\textbf{Case 1: $s=1$}

\medskip

Overlapping $a^pb^{q+1}a^{r+pk}b\to b$ with itself, we obtain a new rule $a^pb^{q+1}a^{r+p(k-1)}b\to b^{q+1}a^{r+pk}b$. Then successively overlapping the newly obtained rules with the initial one, we obtain the following finite complete system for $M$:
\begin{eqnarray*}
a^pb^{q+1}a^{r+pk}b &\to& b\\
a^pb^{q+1}a^{r+pi}b &\to& b^{q+1}a^{r+p(i+1)}b,\quad 0\leq i\leq k-1.
\end{eqnarray*}

\medskip

\noindent\textbf{Case 2: $s>1$ and $r>0$}

\medskip

By the same tactics as in \textbf{Case 1}, we obtain the following finite complete system for $M$:
\begin{eqnarray*}
a^pb^{q+s}a^{r+pk}b^s &\to& b\\
a^pb^{q+s}a^{r+pi}b &\to& b^{q+1}(a^{r+pk}b^{q+2s-1})^{k-1-i}a^{r+pk}b^s,\quad 0\leq i\leq k-1.
\end{eqnarray*}

\medskip

\noindent\textbf{Case 3: $s>1$, $r=0$ and $k=1$}

\medskip

It is easy to see that $M$ admits the following finite complete system:
\begin{eqnarray*}
a^pb^s &\to& x\\
xb^qx &\to& b\\
xb^{q+1} &\to& b^{q+1}x.
\end{eqnarray*}

\medskip

\noindent\textbf{Case 4: $s>1$, $r=0$ and $k\geq 2$}

\medskip

We have the relation $a^pb^{q+s}a^{pk}b^s=b$. We add a new letter $x=a^{pk}b^s$ and then $a^pb^{q+s}x=b$.

Now, $a^{p(k-1)}b=a^{pk}b^{q+s}x=xb^qx$, and so $\underline{a^pxb^qxb^{s-1}=x}$. Since $a^pb\cdot b^{q+s-1}x=b$, we have that $\underline{a^pb}=a^{p(k-1)}b\cdot (b^{q+s-1}x)^{k-2}=\underline{xb^qx(b^{q+s-1}x)^{k-2}}$. Then
\begin{equation*}
\underline{b}=a^pb^{q+s}x=xb^qx(b^{q+s-1}x)^{k-2}\cdot b^{q+s-1}x=\underline{xb^qx(b^{q+s-1}x)^{k-1}}.
\end{equation*}
This yields
\begin{eqnarray*}
\underline{xb^qx(b^{q+s-1}x)^{k-2}b^{q+s}} &=& xb^qx(b^{q+s-1}x)^{k-2}b^{q+s-1}\cdot xb^qx(b^{q+s-1}x)^{k-1}\\
&=& xb^qx(b^{q+s-1}x)^{k-1}\cdot b^qx(b^{q+s-1}x)^{k-1}\\
&=& \underline{b^{q+1}x(b^{q+s-1}x)^{k-1}}.
\end{eqnarray*}
The underlined relations give us the following rewriting system, defining $M$:
\begin{eqnarray*}
a^pxb^qxb^{s-1} &\to& x\\
a^pb &\to& xb^qx(b^{q+s-1}x)^{k-2}\\
xb^qx(b^{q+s-1}x)^{k-1} &\to& b\\
xb^qx(b^{q+s-1}x)^{k-2}b^{q+s} &\to& b^{q+1}x(b^{q+s-1}x)^{k-1}.
\end{eqnarray*}
If $q<s-1$, one readily checks that this system is confluent and terminating (regardless whether $k>2$ or $k=2$).

If $q\geq s-1$, then
\begin{equation*}
a^pxb^{q+1}=a^pxb^q\cdot xb^qx(b^{q+s-1}x)^{k-1}=xb^{q-(s-1)}x(b^{q+s-1}x)^{k-1},
\end{equation*}
and adding the rule
\begin{equation*}
a^pxb^{q+1}\to xb^{q-(s-1)}x(b^{q+s-1}x)^{k-1}
\end{equation*}
to the system, we obtain the required finite complete system.
\end{proof}

\section{Non-Hopfian Example}\label{sec:hopf}

\begin{example}\label{ex:hopf}
The monoid $M=\mathrm{Mon}\langle a,b:ab^2a^2b^2=b\rangle$ is non-hopfian.
\end{example}

\begin{proof}
Our example falls within Case~4 of the proof of
Theorem~\ref{th:abab}. By letting $x=a^2b^2$, we obtain the following
complete system for $M$:
\begin{eqnarray*}
ax^2b &\to& x\\
ab &\to& x^2\\
x^2bx &\to& b\\
x^2b^2 &\to& bxbx.
\end{eqnarray*}
Consider the assignment $a\mapsto a$ and $b\mapsto bab$. Since
\begin{eqnarray*}
a(bab)^2a^2(bab)^2 &\to& x^2\cdot x^2bx^2\cdot ax^2\cdot x^2bx^2\\
&\to& x^2bx\cdot ax^2b\cdot x\\
&\to& bx^2
\end{eqnarray*}
and $bab\to bx^2$, we have that the assignment lifts to a homomorphism. Under this homomorphism $ab^2$ maps to 
\begin{equation*}
abab^2ab\to x^2\cdot x^2bx^2\to x^2bx\to b,
\end{equation*}
and so the homomorphism is surjective. If this homomorphism were
bijective, then we would have that the inverse of this homomorphism
would be a homomorphism given by $a\mapsto a$ and $b\mapsto ab^2$. But
under this assignment the relation $ab^2a^2b^2=b$ does not hold, for:
\begin{eqnarray*}
a\cdot ab^2ab^2a^2\cdot ab^2ab^2 &=& a^2b^2\cdot ab^2\cdot a\cdot a^2b^2\cdot ab^2\\
&=& xab^2axab^2\\
&\to& x^3bax^3b,
\end{eqnarray*}
which does not reduce to $ab^2=x^2b$. Thus $M$ is non-hopfian.
\end{proof}

\begin{remark}
Malcev's Theorem asserts that every finitely presented residually
finite semigroup is hopfian. Thus the monoid $M$ from
Example~\ref{ex:hopf} is non-residually finite. It also follows
immediately from the complete system for $M$ that $M$ has linear Dehn
function.

On the other hand, for groups, it is still an important open question
whether every hyperbolic group is residually finite. (Finitely
presented groups with linear Dehn function are hyperbolic; see
\cite{gromov_hyperbolic}.)
\end{remark}

\section{Remarks and Questions}
\label{sec:final}

We have proved that every monoid $\mathrm{Mon}\langle
a,b:a^{\alpha}b^{\beta}a^{\gamma}b^{\delta}a^{\varepsilon}b^{\varphi}=b\rangle$
admits a finite complete system and will shortly make the proof
available as a preprint. The proof of this result, in comparison to
that of Theorem~\ref{th:abab}, is already very technical and gives
little hope that it is possible to prove that every monoid
$\mathrm{Mon}\langle a,b:aUb=b\rangle$ admits a finite complete system
just by straightforward method. Yet, analysing the cases appearing in
that proof, and looking at the proof of Theorem~\ref{th:abab}, we
noticed that the one-relator monoids under consideration have at most
quadratic Dehn functions and linear space functions. This prompts us
to raise
\begin{problem}
Is it true that
\begin{enumerate}[(1)]
\item
every monoid $\mathrm{Mon}\langle a,b:aUb=b\rangle$ has at most quadratic Dehn function?
\item
every monoid $\mathrm{Mon}\langle a,b:aUb=b\rangle$ has linear space function?
\end{enumerate}
\end{problem}
The reader may wish to consult a brilliant paper of Victor
Guba~\cite{Guba} on some other possible approaches how to deal with monoids
$\mathrm{Mon}\langle a,b:aUb=b\rangle$.

Another question we were trying to settle is whether every monoid
$\mathrm{Mon}\langle
a,b:a^{\alpha}b^{\beta}a^{\gamma}b^{\delta}=b\rangle$ admits a
\emph{length-non-increasing} finite complete system (that is, where
the rewriting rules $l \to r$ are all such that $|l| \geq |r|$). Using
Knuth--Bendix completion in GAP, we have thus far eliminated all our
suspected counterexamples, so we simply ask the general question:

\begin{question}
Do all monoids $\mathrm{Mon}\langle
a,b:a^{\alpha}b^{\beta}a^{\gamma}b^{\delta}=b\rangle$ admit
length-non-increasing finite complete systems?
\end{question}

Note that it follows from the results of G\"{u}nther Bauer and
Friedrich Otto~\cite{Bauer} that there do exist monoids admitting
finite complete systems but not admitting finite complete system which
do not increase the lengths.

\end{document}